\documentclass{amsart}

\usepackage{amsmath}
\usepackage{amssymb}
\usepackage{bibentry}
\usepackage{enumitem}
\usepackage{graphicx}
\usepackage{mathrsfs}
\usepackage{tabu}
\usepackage{tikz}
\usepackage[comma,numbers,sort,square]{natbib}
\usepackage[colorlinks=true,
            linktocpage=true,
            linkcolor=magenta,
            citecolor=magenta,
            urlcolor=magenta]{hyperref}
\PassOptionsToPackage{hyphens}{url}

\usetikzlibrary{arrows}
\usetikzlibrary{shapes}
\usetikzlibrary{snakes}

\newtheorem{theorem}{Theorem}[section]

\newtheorem{lemma}[theorem]{Lemma}

\theoremstyle{definition}
\newtheorem{definition}[theorem]{Definition}

\newtheorem{question}[theorem]{Question}

\newtheorem*{note*}{Note}

\DeclareMathOperator{\res}{\upharpoonright}
\newcommand{\ran}{\operatorname{ran}}

\newcommand{\seq}[1]{\langle #1 \rangle}

\newcommand{\RT}{\mathsf{RT}}
\newcommand{\COH}{\mathsf{COH}}

\newcommand{\SRT}{\mathsf{SRT}}

\newcommand{\D}{\mathsf{D}}

\newcommand{\cequiv}{\equiv_{\text{\upshape c}}}

\newcommand{\nscred}{\nleq_{\text{\upshape sc}}}

\newcommand{\ured}{\leq_{\text{\upshape W}}}
\newcommand{\nured}{\nleq_{\text{\upshape W}}}
\newcommand{\suequiv}{\equiv_{\text{\upshape sW}}}
\newcommand{\sured}{\leq_{\text{\upshape sW}}}

\newcommand{\Tred}{\leq_{\text{\upshape T}}}

\newcommand{\unlock}{\mathrm{u}}
\newcommand{\name}[1]{\dot{#1}}
\newcommand{\Pforcing}{\mathbb{P}}
\newcommand{\Qforcing}{\mathbb{Q}}
\newcommand{\Rforcing}{\mathbb{R}}

\begin{document}

\title{$\COH$, $\SRT^2_2$, and multiple functionals}

\author{Damir D. Dzhafarov}
\address{Department of Mathematics\\
University of Connecticut\\
341 Mansfield Road\\ Storrs, Connecticut 06269-1009 U.S.A.}
\email{damir@math.uconn.edu}

\author{Ludovic Patey}
\address{Institut Camille Jordan\\
Universit\'{e} Claude Bernard Lyon 1\\
43 boulevard du 11 novembre 1918, F-69622 Villeurbanne Cedex, France}
\email{ludovic.patey@computability.fr}

\thanks{Dzhafarov was supported by a Collaboration Grant for Mathematicians from the Simons Foundation and by a Focused Research Group grant from the National Science Foundation of the United States, DMS-1854355. Patey was partially supported by grant ANR “ACTC” \#ANR-19-CE48-0012-01. The authors thank the anonymous referee for helpful comments.}
\begin{abstract}
	We prove the following result: there is a family $R = \seq{R_0,R_1,\ldots}$ of subsets of $\omega$ such that for every stable coloring $c : [\omega]^2 \to k$ hyperarithmetical in $R$ and every finite collection of Turing functionals, there is an infinite homogeneous set $H$ for $c$ such that none of the finitely many functionals map $R \oplus H$ to an infinite cohesive set for $R$. This provides a partial answer to a question in computable combinatorics, whether $\COH$ is omnisciently computably reducible to $\SRT^2_2$.
\end{abstract}

\maketitle

\section{Introduction}

The \emph{$\SRT^2_2$ vs.\ $\COH$ problem} is a question in computable combinatorics that aims to clarify the relationship between two well-studied combinatorial consequences of Ramsey's theorem for pairs in terms of their effective content. Recently solved in its original form by Monin and Patey \cite{MP-TA2}, it has given rise to several related and more general problems, as we detail further below. In this article, we establish a new partial results towards the resolution of one of the principal outstanding questions in this inquiry.

For completeness, and also to fix some notation, we begin by briefly reviewing the most relevant definitions below. We refer the reader to Hirschfeldt \cite[Chapter 6]{Hirschfeldt-2014} for a more thorough discussion and overview of computable combinatorics. We assume familiarity with computability theory and reverse mathematics, and refer to Soare \cite{Soare-2016} and Simpson \cite{Simpson-2009}, respectively, for background on these subjects. We also assume the basics of Weihrauch reducibility and computable reducibility, and refer, e.g., to Brattka, Gherardi, and Pauly~\cite{BGP-TA} for a detailed survey, or, e.g., to Cholak, Dzhafarov, Hirschfeldt, and Patey~\cite[Section 1]{CDHP-TA} for an introduction aimed more specifically at the kinds of questions we will be dealing with here.

\begin{definition}
	Fix numbers $n,k \geq 1$.
	\begin{enumerate}
		\item For every set $X \subseteq \omega$, let $[X]^n = \{ \seq{x_0,\ldots,x_{n-1}} \in \omega^n : x_0 < \cdots < x_{n-1}\}$.
		\item A \emph{$k$-coloring of $[\omega]^n$} is a map $c : [\omega]^n \to \{0,\ldots,k-1\}$.
		\item A set $H \subseteq \omega$ is \emph{homogeneous} for $c$ if $c \res [H]^n$ is constant.
		\item A \emph{$k$-coloring of $[\omega]^2$} is \emph{stable} if $\lim_y c(\seq{x,y})$ exists for all $x \in \omega$.
		\item A set $L \subseteq \omega$ is \emph{limit-homogeneous} for a stable $c : [\omega]^2 \to k$ if $\lim_y c(x,y)$ is the same for all $x \in L$.
	\end{enumerate}
\end{definition}

\noindent When $n = 2$, we call $c : [\omega]^2 \to k$ a \emph{$k$-coloring of pairs}, or simply a \emph{coloring of pairs} if $k$ is understood. We will write $c(x,y)$ in place of $c(\seq{x,y})$.

The following definition is somewhat nonstandard and technical, but it will simplify the presentation in the sequel.

\begin{definition}\label{D:coh}
Let $R = \seq{r_0,r_1,\ldots}$ be a family of functions $r_i : \omega \to \omega$.
	\begin{enumerate}
		\item $R$ is a \emph{bounded family of functions} if for all $n$ there is a $k$ so that $\ran(r_n) < k$.
		\item For $k \in \omega$, $R$ is a \emph{$k$-bounded family of functions} if $r_n(x) < k$ for all $n$ and~$x$.
		\item A set $X$ is \emph{cohesive} for $R$ if for each $n$ there is a $y \in \omega$ such that $f_n(x) = y$ for all but finitely many $x \in X$.
	\end{enumerate}
\end{definition}

\noindent The more typical definition of cohesiveness is with respect to a family $\seq{R_0,R_1,\ldots}$ of subsets of $\omega$, for which a set $X$ is cohesive if for each $n$, either $X \cap R_n$ or $X \cap \overline{R_n}$ is finite. Of course, if we identify sets with their characteristic functions then we see that this is just the same as being cohesive for a $2$-bounded family of functions. We return to this below.

We follow the now-standard practice of regarding $\Pi^1_2$ statements of second-order arithmetic as \emph{problems}, equipped with a set of instances, and for each instance, a set of solutions, all coded or represented by subsets of $2^\omega$ (see \cite{CDHP-TA}, Definition 1.1). This facilitates their study both in the framework of reverse mathematics and in terms of Weihrauch and computable reducibilities. We shall not be explicit about this identification moving forward, as it is obvious for all of the principles we will be looking at. These are the following.

\begin{definition}\label{D:principles}
\
	\begin{enumerate}
		\item \emph{Ramsey's theorem} is the statement that for all $n,k \geq 1$, every $c : [\omega]^n \to k$ has an infinite homogeneous set.
		\item \emph{Stable Ramsey's theorem for pairs}, denoted $\SRT^2_{<\infty}$, is the restriction of Ramsey's theorem to stable colorings of pairs.
		\item \emph{The $\Delta^0_2$ subset principle}, denoted $\D^2_{<\infty}$, is the statement for all $k \geq 1$, every stable $c : [\omega]^2 \to k$ has an infinite limit-homogeneous set.
		\item \emph{The cohesiveness principle for bounded families}, denoted $\COH_\omega$, is the principle that every bounded family of functions has an infinite cohesive set.
		\item For fixed $n,k \geq 1$, $\RT^n_k$ denotes the restriction of Ramsey's theorem to $k$-colorings of $[\omega]^n$.
		\item For fixed $k \geq 1$, $\SRT^2_k$ and $\D^2_k$ denote the restrictions of $\SRT^2_{<\infty}$ and $\D^2_{<\infty}$, respectively, to $k$-colorings.
		\item For fixed $k \geq 1$, $\COH_k$ is the restriction of $\COH_\omega$ to $k$-bounded families of functions.
	\end{enumerate}
\end{definition}

\noindent For $n = 2$, the traditional notation for $\COH_2$ is $\COH$, and we shall follow this below. However, we can really use the various restrictions of $\COH_\omega$ defined above interchangeably, as the following lemma shows.

\begin{lemma}\label{L:cohequiv}
	For all $k \geq 2$, we have $\COH \suequiv \COH_k \suequiv \COH_\omega$.
\end{lemma}

\begin{proof}
	Obviously, $\COH \sured \COH_k \sured \COH_\omega$. It remains only to show that $\COH_\omega \sured \COH$. For all $k, y \in \omega$, let $y_k$ be $y$ written in binary, either truncated or prepended by $0s$ to have exactly $\ulcorner \log_2 k \urcorner$ many digits. We view $y_k$ as a string, and write $y_k(i)$ for its $i$th digit. Now fix a bounded family of functions $R = \seq{r_0,r_1,\ldots}$. Let $b : \omega \to \omega$ be the function $b(n) = (\mu k)(\forall x)[r_n(x) < k]$ for all $n \in \omega$. Then $b$ is uniformly $R'$-computable. So we can fix a uniformly $R$-computable approximation $\widehat{b} : \omega^2 \to \omega$ to $b$, so that $\lim_s \widehat{b}(n,s) = b(n)$ for all $n$. Define a $2$-bounded family of functions $R' = \seq{r'_0,r'_1,\ldots}$ as follows: for all $m,x \in \omega$,
	\[
		r'_m(x) =
		\begin{cases}
			r_n(x)_{\widehat{b}(n,s)}(i) & \text{if } (\exists n,s \in \omega)(\exists i < \ulcorner \log_2 \widehat{b}(n,s) \urcorner)~m = \seq{\widehat{b}(n,s),i} \\
			0 & \text{otherwise.}
		\end{cases}
	\]
	Then $R'$ is a uniformly $R$-computable, and it is not difficult to see that every infinite cohesive set for $R'$ is also cohesive for $R$. This completes the proof.
\end{proof}

\noindent A well-known fact about $\COH$ (in the parlance of Definitions \ref{D:coh} and \ref{D:principles}) is that if $X$ computes an infinite cohesive set for some $2$-bounded family of functions $R = \seq{r_0,r_1,\ldots}$, then so does any set $Y$ satisfying $R \Tred Y$ and $X' \Tred Y'$. By the preceding lemma, we see that the same holds for any bounded family of functions.

The relationship between the stable Ramsey's theorem and the cohesiveness principle is the focus of a longstanding and ongoing investigation (see, e.g.,~\cite{CDHP-TA, DGHPP-TA, BR-2017, DDHMS-2016, Dzhafarov-2016, DPSW-2017, HJ-2016, HM-TA, MP-TA, Patey-2016c, Patey-2016}). We refer the reader to~\cite[Section 1]{CDHP-TA} for a discussion of some of the history of these principles, and their larger significance in the exploration of the logical strength of combinatorial principles. For many years, the central open question in this investigation was the so-called \emph{$\SRT^2_2$ vs.\ $\COH$ problem}, which asked whether every $\omega$-model of $\SRT^2_2$ also satisfy $\COH$. The answer was recently shown to be no.

\begin{theorem}[Monin and Patey \cite{MP-TA2}]\label{thm:cohsrt22}
	There exists an $\omega$-model of $\SRT^2_2$ in which $\COH$ fails.	
\end{theorem}

\noindent The quest to obtain this solution, by multiple authors, gave rise many related questions, many of which hint more deeply at the combinatorial nature of the relationship between $\SRT^2_2$ and $\COH$, and which remain open. Our focus in this paper will be on a question that has arguably emerged as the most central among these. We first recall the definition of omniscient reducibility, introduced by Monin and Patey \cite[Section 1.1]{MP-TA}.

\begin{definition}
	Let $\mathsf{P}$ and $\mathsf{Q}$ be problems. 
	\begin{enumerate}
		\item $\mathsf{P}$ is \emph{omnisciently computably reducible to} $\mathsf{Q}$ if for every $\mathsf{P}$-instance $X$ there is a $\mathsf{Q}$-instance $\widehat{X}$ with the property that if $\widehat{Y}$ is any $\mathsf{Q}$-solution to $\widehat{X}$ then $X \oplus \widehat{Y}$ computes a $\mathsf{P}$-solution to $X$.
		\item $\mathsf{P}$ is \emph{omnisciently Weihrauch reducible to} $\mathsf{Q}$ if there is a Turing functional $\Psi$ such that for every $\mathsf{P}$-instance $X$ there is a $\mathsf{Q}$-instance $\widehat{X}$ with the property that if $\widehat{Y}$ is any $\mathsf{Q}$-solution to $\widehat{X}$ then $\Psi(X \oplus \widehat{Y})$ is a $\mathsf{P}$-solution to $X$.
	\end{enumerate}
	The reductions above are \emph{strong} if the relevant computation of a $\mathsf{P}$-solution to $X$ works with just $\widehat{Y}$ as an oracle, rather than $X \oplus \widehat{Y}$.
\end{definition}



\begin{question}\label{Q:omni}
	Is $\COH$ omnisciently computably reducible to $\D^2_2$, or to $\D^2_{<\infty}$?
\end{question}

We can at first compare this to the question of whether $\COH$ is simply computably reducible to $\D^2_{<\infty}$ (or $\D^2_k$ for some $k$). Here, the answer is no. Indeed, it is easy to see that $\SRT^2_{<\infty} \cequiv \D^2_{<\infty}$, and that for each specific $k$, also $\SRT^2_k \cequiv \D^2_k$. Thus, in the computable reducibility question we could replace $\D^2_{<\infty}$ by $\SRT^2_{<\infty}$ (or $\D^2_k$ by $\SRT^2_k$), and then the answer follows by Theorem \ref{thm:cohsrt22} since computable reducibility implies implication over $\omega$-models. (Alternatively, it is easy to see that over $\omega$-models, $\D^2_k$, $\D^2_{<\infty}$, $\SRT^2_k$, and $\SRT^2_{<\infty}$ are equivalent, for all $k \geq 2$.)

Omniscient computable reducibility is more sensitive. While Question \ref{Q:omni} is open, if we replace $\D^2_2$ by $\SRT^2_2$ there then the answer is known: $\COH$ \emph{is} omnisciently computably reducible even to $\SRT^2_2$ (see \cite{CDHP-TA}, Proposition 2.2). On the other hand, we can replace $\D^2_2$ by $\RT^1_2$, as these are easily seen to be omnisciently computably equivalent, and similarly for $\D^2_{<\infty}$ and $\RT^1_{<\infty}$. Here, it will be easier to work with $\D^2_k$ and $\D^2_{<\infty}$, so the rest of our discussion is formulated in terms of these principles.

For completeness, we note also that Dzhafarov \cite[Theorem 3.2 and Corollary 3.5]{Dzhafarov-2016} showed that $\SRT^2_2$ is not omnisciently Weihrauch, or strongly omnisciently computably, reducible to $\D^2_{<\infty}$, while Patey \cite[Corollary 3.3]{Patey-2016} showed that for all $k > \ell \geq 1$, $\D^2_k$ is not strongly omnisciently computably reducible to $\SRT^2_\ell$. Thus, the relationships between different versions of the stable Ramsey's theorem and the $\Delta^0_2$ subset principle in terms of known reducibilities are fully understood.

As described in \cite[Sections 1 and 2]{CDHP-TA}, Question~\ref{Q:omni} seems to encompass the true combinatorial core of the relationship between cohesiveness and homogeneity. Adapting the techniques from Monin and Patey's resolution of the $\SRT^2_2$ vs.\ $\COH$ problem, or the techniques from earlier, partial solutions by Dzhafarov \cite{Dzhafarov-2016} and Dzhafarov, Patey, Solomon, and Westrick \cite{DPSW-2017} (who established that $\COH \nured \SRT^2_{<\infty}$ and $\COH \nscred \SRT^2_{<\infty}$, respectively) has so far proved difficult. There is thus a wide gap between the current results and Question \ref{Q:omni}. Our approach here is to narrow this gap by allowing for multiple functionals in the ``backward'' direction. For succinctness, we introduce the following definition:

\begin{definition}
	Let $\mathsf{P}$ and $\mathsf{Q}$ be problems.
	\begin{enumerate}
		\item $\mathsf{P}$ is \emph{Weihrauch reducible to $\mathsf{Q}$ with finitely many functionals} if there is a Turing functional $\Phi$ such that for every $\mathsf{P}$-instance $X$ there is a finite set of Turing functionals $\Psi_0,\ldots,\Psi_{t-1}$ such that $\Phi(X)$ is a $\mathsf{Q}$-instance and if $\widehat{Y}$ is any $\mathsf{Q}$-solution to $\Phi(X)$ then there is a $t < s$ with $\Psi_t(X \oplus \widehat{Y})$ a $\mathsf{P}$-solution to~$X$.
		\item $\mathsf{P}$ is \emph{computably reducible to $\mathsf{Q}$ with finitely many functionals} if for every $\mathsf{P}$-instance $X$ there is a $\mathsf{Q}$-instance $\widehat{X} \Tred X$ and a finite set of Turing functionals $\Psi_0,\ldots,\Psi_{t-1}$ such that if $\widehat{Y}$ is any $\mathsf{Q}$-solution to $\widehat{X}$ then there is a $t < s$ with $\Psi_t(X \oplus \widehat{Y})$ a $\mathsf{P}$-solution to $X$.
		\item $\mathsf{P}$ is \emph{hyperarithmetically computably reducible to $\mathsf{Q}$ with finitely many functionals} if for every $\mathsf{P}$-instance $X$ there is a $\mathsf{Q}$-instance $\widehat{X}$ hyperarithmetical in $X$ and a finite set of Turing functionals $\Psi_0,\ldots,\Psi_{t-1}$ such that if $\widehat{Y}$ is any $\mathsf{Q}$-solution to $\widehat{X}$ then there is a $t < s$ with $\Psi_t(X \oplus \widehat{Y})$ a $\mathsf{P}$-solution to $X$.
		\item $\mathsf{P}$ is \emph{omnisciently computably reducible to $\mathsf{Q}$ with finitely many functionals} if for every $\mathsf{P}$-instance $X$ there is a $\mathsf{Q}$-instance $\widehat{X}$ and a finite set of Turing functionals $\Psi_0,\ldots,\Psi_{t-1}$ such that if $\widehat{Y}$ is any $\mathsf{Q}$-solution to $\widehat{X}$ then there is a $t < s$ with $\Psi_t(X \oplus \widehat{Y})$ a $\mathsf{P}$-solution to $X$.
	\end{enumerate}
\end{definition}

\noindent The basic relationships between the above reducibilities are as follows:
\[
\begin{array}{lll}
	\mathsf{P} \ured \mathsf{Q} & \implies & \mathsf{P} \text{ is Weihrauch reducible to } \mathsf{Q} \text{ with finitely many functionals}\\\\\
	& \implies & \mathsf{P} \text{ is computably reducible to } \mathsf{Q} \text{ with finitely many functionals}\\\\
	& \implies & \mathsf{P} \text{ is hyperarithmetically reducible to } \mathsf{Q} \text{ with finitely many}\\
	& & \text{functionals}\\\\
	& \implies & \mathsf{P} \text{ is omnisciently computably reducible to } \mathsf{Q} \text{ with finitely mamy}\\
	& & \text{functionals}\\\\
	& \implies & \mathsf{P} \text{ is omnisciently computably reducible to } \mathsf{Q}.
\end{array}
\]
\noindent Note also that while Weihrauch reducibility with finitely many functionals is a generalization of Weihrauch reducibility, computable reducibility with finitely many functionals is a restriction of computable reducibility.  A good example here is to look at $\SRT^2_2$ and $\D^2_2$: as mentioned, $\SRT^2_2 \nured \D^2_2$, but it is easy to see that $\SRT^2_2$ is Weihrauch reducible to $\D^2_2$ with finitely many (in fact, two) functionals. We can now state our main result:

\begin{theorem}\label{T:main}
	$\COH$ is not hyperarithmetically computably reducible to $\D^2_{<\infty}$
	with finitely many functionals.
\end{theorem}

\noindent  That is, we build a family of sets $G = \seq{G_0,G_1,\ldots}$ such that for every stable coloring hyperarithmetical in $G$ and every finite collection of Turing functionals $\Psi_0,\ldots,\Psi_{s-1}$, there exists an infinite limit-homogeneous set $H$ for $c$ such that $\Psi_t(G \oplus H)$ is not an infinite cohesive set for $G$, for any $t < s$.

Our construction will force the instance $G$ of $\COH$ to be non-hyperarithmetical. We do not know whether, in general, this is necessary, or whether there exists a witness to Theorem \ref{T:main} that is hyperarithmetical, or perhaps arithemtical or even computable. Indeed, it is even possible that there is a computable instance of $\COH$ witnessing a negative answer to Question \ref{Q:omni}.

The rest of this paper is dedicated to a proof of Theorem \ref{T:main}. For ease of understanding, we organize this into two parts. In Section \ref{S:construction} we present a proof just for the case of stable $2$-colorings. Then, in Section \ref{S:extensions}, we explain how the argument can be adapted to obtain the theorem in its full generality.

\begin{note*}
	The proof of Monin and Patey \cite{MP-TA2} of Theorem \ref{thm:cohsrt22} was announced after the original submission of this article. We have rewritten the introduction here to reflect this fact.
\end{note*}

\section{Construction}\label{S:construction}

Our approach uses an elaboration on the forcing methods introduced by Dzhafarov \cite{Dzhafarov-2015} for building instances of $\COH$, and by Cholak, Jockusch, and Slaman \cite[Section 4]{CJS-2001} for building solutions to $\D^2_2$. With respect to the latter, our proof here has a crucial innovation. As in other applications, we force with Mathias conditions, defined below. But here, our reservoirs are not computable or low, or indeed absolute sets of any other kind. Rather, they are names for sets in the forcing language we use to build our $\COH$ instance. This allows us to control not just the $\COH$ instance and the $\D^2_2$ solution separately, as is done, e.g., in \cite{Dzhafarov-2015} or \cite{DPSW-2017}, but also to control their join. We refer the reader to Shore \cite[Chapter 3]{Shore-2016} and Sacks \cite[Section IV.3]{Sacks-1990} for background on forcing in arithmetic, and the latter specifically for an introduction to forcing over the hyperarithmetic hierarchy.

In what follows, several notions of forcing are defined. When no confusion can arise, we refer to the conditions and extension relation in each of these simply as ``conditions'' and ``extension'', without explicitly labeling these by the forcing itself.

\subsection{Generic instances of $\COH$}

\begin{definition}
	Let $\Pforcing$ be the notion of forcing whose conditions are tuples $p = (\sigma_0,\ldots,\sigma_{|p|-1},f)$ as follows:
	\begin{itemize}
		\item $|p| \in \omega$;
		\item $\sigma_n \in 3^{<\omega}$ for each $n < |p|$;
		\item $f$ is a function $|p| \to 3 \cup \{\unlock\}$.
	\end{itemize}
	A condition $q = (\tau_0,\ldots,\tau_{|q|-1},g)$ extends $p$, written $q \leq p$, if:
	\begin{itemize}
		\item $|p| \leq |q|$;
		\item $f \preceq g$;
		\item $\sigma_n \preceq \tau_n$ for all $n < |p|$;
		\item if $f(n) \neq \unlock$ for some $n < |p|$ then $\tau_n(x) = f(n)$ for all $x \in [|\sigma_n|,|\tau_n|)$.
	\end{itemize}
\end{definition}

Given a $\Pforcing$-condition $p = (\sigma_0,\ldots,\sigma_{|p|-1},f)$, we also write $\sigma_n^p$ and $f^p$ for $\sigma_n$ and $f$, respectively. If $\mathcal{G}$ is a sufficiently generic filter on $\Pforcing$ then we can define
\[
	G^{\mathcal{G}}_n = \bigcup_{p \in \mathcal{G}, |p| > n} \sigma^p_n
\]
and $G^{\mathcal{G}} = \bigoplus_{n \in \omega} G^\mathcal{G}_n$. Note that this is an instance of $\COH_3$, and that by genericity, there are infinitely many $n$ such that $\lim_x G^{\mathcal{G}}_n(x)$ exists, and infinitely many $n$ such that $\lim_x G^{\mathcal{G}}_n(x)$ does not exist.

The $\Pforcing$ forcing language and forcing relation are defined inductively as usual, and we use $\name{G}_n$ and $\name{G}$ as names for $G^{\mathcal{G}}_n$ and $G^{\mathcal{G}}$. More generally, we help ourselves to names (or \emph{$\mathbb{P}$-names}) for all definable sets in the forcing language and use these as parameters in other definitions.

\begin{lemma}\label{L:lockspres}
	Let $\varphi(\name{G})$ be a $\Sigma^0_2(\name{G})$ formula in the forcing language that is forced by some condition $p$. Let $q$ be the condition that is the same as $p$, only there is some $n < |p|$ such that $f^p(n) = \unlock$ and $f^q(n) \neq \unlock$. Then $q$ forces $\varphi(\name{G})$.
\end{lemma}

\begin{proof}
	As we are employing strong forcing, it suffices 	to consider the case that $\varphi(\name{G})$ is $\Pi^0_1(\name{G})$. Thus, $\varphi(\name{G})$ can be put in the form $\neg (\exists x)\psi(\name{G},x)$, where $\psi$ has only bounded quantifiers and has no free variables other than $x$. If $q$ does not force this formula then by definition there is some $r \leq q$ and some $a \in \omega$ such that $r$ forces $\psi(\name{G},a)$. Now, as $\Psi(\name{G},a)$ has no free variables, it can be put in quantifier-free conjunctive normal form. But the fact that each clause in this conjunction is forced by $r$ depends only on the strings $\sigma^r_0,\ldots,\sigma^r_{|r|-1}$. So let $r'$ be the condition that is the same as $r$, except that $f^{r'}(n) = f^p(n) = \unlock$. Then $r'$ still forces $\Psi(\name{G},a)$, and hence also $(\exists x)\psi(\name{G},x)$. But $r'$ is an extension of $p$, and hence witnesses that $p$ could not force $\neg (\exists x)\psi(\name{G},x)$ or $\varphi(\name{G})$, a contradiction.
\end{proof}

\begin{lemma}\label{L:nolow}
	If $\mathcal{G}$ is a sufficiently generic filter on $\Pforcing$ then there is no infinite cohesive set for $G^{\mathcal{G}}$ which is low over $G^{\mathcal{G}}$.	
\end{lemma}

\begin{proof}
	By the remark following Lemma \ref{L:cohequiv}, it suffices to show that $G^{\mathcal{G}}$ has no $G^{\mathcal{G}}$-computable infinite cohesive set. Fix any functional $\Delta$, and any condition $p$. We exhibit an extension of $p$ forcing that $\Delta(\name{G})$ is not an infinite cohesive set for $\name{G}$. This density fact and the genericity of $\mathcal{G}$ will yield the lemma. Let $n = |p|$. Let $q$ be any extension of $p$ with $|q| = n + 1$ and $f^{q}(n) = \unlock$. If $q$ forces that for each $i < 3$ and each $z \in \omega$ there is an $x > z$ such that $\Delta(\name{G})(x) \downarrow = 1$ and $\name{G}_n(x) = i$, then we can take $q$ to be the desired extension. So suppose otherwise. Then there is an $i < 3$, a $z \in \omega$, and an $r \leq q$ such that no extension of $r$ forces that there is an $x > z$ with $\Delta(\name{G})(x) \downarrow = 1$ and $\name{G}_n(x) = i$. In this case, let $s$ be the condition that is the same as $r$, except that $f^s(n) = i$. Then $s \leq p$ and forces that for all $x > \max\{x,|\sigma^s_n|\}$ we have $\Delta(\name{G})(x) \simeq 0$.
\end{proof}

In the next section, we assemble the pieces to diagonalize all stable colorings hyperarithmetical in our generic instance $G^\mathcal{G}$ of $\COH$. We formualte the pieces for a fixed hyperarithmetical operator $\Gamma$, and then apply them across all such operators in the final proof. This forces our filter $\mathcal{G}$ to be hyperarithmetically generic, and hence, as remarked earlier, $G^\mathcal{G}$ to be non-hyperarithmetical.
\subsection{Generic limit-homogeneous sets}\label{S:Qforcing}

Throughout this section, let $\Gamma$ be a fixed hyperarithmetical operator, and let $\Psi_0,\ldots,\Psi_{s-1}$ be fixed Turing functionals. Let $p_\Gamma$ be a fixed $\Pforcing$-condition forcing that $\Gamma(\name{G})$ is a stable coloring $[\omega]^2 \to 2$ with no infinite limit-homogeneous set which is low over $\dot{G}$. For each $i < 2$ we let $\name{A}_i$ be a name for the set $\{x \in \omega: \lim_y \Gamma(\name{G})(x,y) = i\}$.

\begin{definition}
	Let $\Qforcing_{p_\Gamma}$ be the notion of forcing whose conditions are tuples $(p,D_0,D_1,\name{I})$ as follows:
	\begin{itemize}
		\item $p$ is a $\Pforcing$-condition extending $p_\Gamma$;
		\item $D_i$ is a finite set for each $i < 2$, and $p$ forces that $D_i \subseteq \name{A}_i$;
		\item $\name{I}$ is a $\Pforcing$-name, and $p$ forces that $\name{I}$ is an infinite set which is low over $\dot{G}$, and $\max D_0 \cup D_1 < \min \name{I}$.	
	\end{itemize}
	A condition $(q,E_0,E_1,\name{J})$ extends $(p,D_0,D_1,\name{I})$ if:
	\begin{itemize}
		\item $q \leq p$;
		\item $D_i \subseteq E_i$ for each $i < 2$;
		\item $q$ forces that $E_i \smallsetminus D_i \subseteq \name{I}$ for each $i < 2$, and that $\name{J} \subseteq \name{I}$.
	\end{itemize}
\end{definition}

\noindent Thus, we can think of $\mathbb{Q}_{p_\Gamma}$-condition as $p$, together with a pair of Mathias conditions, $(D_0,\name{I})$ and $(D_1,\name{I})$, that share a common reservoir.

For the remainder of this section, let $\Psi_0,\ldots,\Psi_{s-1}$ be a fixed collection of Turing functionals.

\begin{lemma}\label{L:Mexists}
	The collection of $\Pforcing$-conditions $p^*$ with the following property is dense below $p_\Gamma$: there exists a $\mathbb{Q}_{p_\Gamma}$-condition $(p^*,D_0^*,D_1^*,\name{I})*$ and a maximal subset $M$ of $2 \times s$ such that for all $\seq{i,t} \in M$, $p^*$ forces that there is a $z \in \omega$ such that $\Psi_t(\name{G} \oplus (D_i^* \cup F))(x) \simeq 0$ for all finite sets $F \subseteq \name{I}^*$ and all $x > z$.
\end{lemma}

\begin{proof}
	Let $p \leq p_\Gamma$ be given. We exhibit a $p^*$ as above below $p$. Fix an enumeration of all pairs $\seq{i,t} \in 2 \times s$. Define $M_0 = \emptyset$, and and let $(p^0,D^0_0,D^0_1,\name{I}^0)$ be the $\Qforcing_{p_\Gamma}$-condition $(p,\emptyset,\emptyset,\name{\omega})$. By induction, suppose that we have defined $M_k \subseteq 2 \times s$ for some $k < 2s$, along with some $\Qforcing_{p_\Gamma}$-condition $(p^k,D^k_0,D^k_1,\name{I}^k)$. Let $\seq{i,t}$ be the $(k+1)$-st element of our enumeration of $2 \times s$. If there is a condition $(q,E_0,E_1,\name{J})$ extending $(p^k,D^k_0,D^k_1,\name{I}^k)$ such that $q$ forces there is a $z \in \omega$ such that $\Psi_t(\name{G} \oplus (E_i \cup F))(x) \simeq 0$ for all finite sets $F \subseteq \name{J}$ and all $x > z$, let $M_{k+1} = M_k \cup \{\seq{i,t}\}$ and let $(p^{k+1},D^{k+1}_0,D^{k+1}_1,\name{I}^{k+1})$ be such a $(q,E_0,E_1,\name{J})$. Otherwise, let $M_{k+1} = M_k$ and let $(p^{k+1},D^{k+1}_0,D^{k+1}_1,\name{I}^{k+1}) = (p^k,D^k_0,D^k_1,\name{I}^k)$. Clearly, $M = M_{2s}$ and $(p^*,D_0^*,D_1^*,\name{I}^*) = (p^{2s},D^{2s}_0,D^{2s}_1,\name{I}^{2s})$ satisfy the conclusion of the lemma.
\end{proof}

For the duration of this section, let $(p^*,D_0^*,D_1^*,\name{I}^*)$ and $M$ as above be fixed.

\begin{definition}\label{D:Rforcing}
	Let $\Rforcing_{p^*,D_0^*,D_1^*,\name{I}^*}$ be the restriction of $\Qforcing_{p_\Gamma}$ to conditions extending $(p^*,D_0^*,D_1^*,\name{I}^*)$ of the form $(p,D_0,D_1,\name{I}^* \cap [u,\infty))$.
\end{definition}

\noindent To visually distinguish $\Rforcing_{p^*,D_0^*,D_1^*,\name{I}^*}$-conditions from more general $\Qforcing_{p_\Gamma}$-extensions of $(p^*,D_0^*,D_1^*,\name{I}^*)$, we denote the $\Rforcing_{p^*,D_0^*,D_1^*,\name{I}^*}$-condition $(p,D_0,D_1,\name{I}^* \cap [u,\infty))$ by $(p,D_0,D_1,u)$. Note that $(p^*,D_0^*,D_1^*,\name{I}^*)$ is of course an $\Rforcing_{p^*,D_0^*,D_1^*,\name{I}^*}$-condition.

We now assemble a couple of density facts that we will use to prove our theorem.

\begin{lemma}\label{L:infsides}
	Let $(p,D_0,D_1,u)$ be an $\Rforcing_{p^*,D_0^*,D_1^*,\name{I}^*}$-condition. The collection of $\Pforcing$-condition $q$ for which there exists an $\Rforcing_{p^*,D_0^*,D_1^*,\name{I}^*}$-condition $(q,E_0,E_1,v)$ extending $(p,D_0,D_1,u)$, and satisfying $|E_i| = |D_i| + 1$ for each $i < 2$, is dense below $p$.
\end{lemma}

\begin{proof}
	Fix any $r \leq p$. Let $q$ be any extension of $r$ deciding, for each $i < 2$, if there is an $x \geq u$ in $\name{I}^* \cap \name{A}_i$. If for some $i < 2$, $q$ forces that there is no such $x$, then $q$ forces that $\name{I}^* \cap [u,\infty) \subseteq A_{1-i}$. But as $q \leq p^*$, we have that $q$ forces that $\name{I}^*$ is an infinite set which is low over $\name{G}$, and hence that $\name{I}^* \cap [u,\infty)$ is an infinite set which is low over $\name{G}$. But by assumption, $p_\Gamma$ forces that there is no such set contained in $\name{A}_{1-i}$, so since $q \leq p_\Gamma$ this is a contradiction. Thus, it must be that $q$ forces, for each $i < 2$, that there is an $x \geq u$ in $\name{I}^* \cap \name{A}_i$. We can thus fix an $x_i \geq u$ for each $i < 2$ such that $q$ forces that $x_i \in \name{I}^* \cap \name{A}_i$. Let $E_i = D_i \cup \{x_i\}$ for each $i$, and let $v = \max \{x_0,x_1\} + 1$; then $(q,E_0,E_1,v)$ witnesses that $q$ is the desired extension.
\end{proof}

The next lemma facilitates the crucial step of reflecting a $\Rforcing_{p^*,D_0^*,D_1^*,\name{I}^*}$-condition into $\mathbb{P}$.

\begin{lemma}\label{L:progress}
	Let $(p,D_0,D_1,u)$ be an $\Rforcing_{p^*,D_0^*,D_1^*,\name{I}^*}$-condition, and assume that $f^p(n) = \unlock$ for some $n \in [|p^*|, |p|)$. For all $z \in \omega$, $j < 3$, and $\seq{0,t_0}, \seq{1,t_1} \in 2 \times s \smallsetminus M$, the collection of $\mathbb{P}$-conditions $q$ with the following property is dense below $p$: there exists an $\Rforcing_{p^*,D_0^*,D_1^*,\name{I}^*}$-condition $(q,E_0,E_1,v)$ extending $(p,D_0,D_1,u)$ and numbers $i < 2$ and $x > z$ such that $q$ forces that $\Psi_{t_i}(\name{G} \oplus E_i)(x) \downarrow = 1$ and $\name{G}_n(x) = j$.
\end{lemma}

\begin{proof}
	Fix any $r \leq p$. Consider the $\Pi^0_1(\name{G},\name{I}^*)$ formula $\psi(\name{G},\name{I}^*,X_0,X_1)$ of two set variables asserting:
	\begin{itemize}
		\item $X_0$ and $X_1$ partition $\name{I}^* \cap [u,\infty)$;
		\item for each $i < 2$, each $x > z$, and each finite set $F \subseteq X_i$ it is not the case that $\Psi_{t_i}(\name{G} \oplus (D_i \cup F))(x) \downarrow = 1$ and $\name{G}_{n}(x) = j$.
	\end{itemize}
	Let $\varphi(\name{G},\name{I}^*)$ be the formula $(\exists X_0,X_1)\psi(\name{G},\name{I}^*,X_0,X_1)$. Then $\varphi(\name{G},\name{I}^*)$ is also $\Pi^0_1(\name{G},\name{I}^*)$, and we can thus fix some $\widehat{r} \leq r$ that decides this formula.
	
	Suppose first that $\widehat{r}$ forces $\varphi(\name{G},\name{I}^*)$. Let $\widehat{r}'$ be the condition that is the same as $\widehat{r}$ except that $f^{\widehat{r}'}(n) = j$ for each $i < 2$. We claim that $\widehat{r}'$ forces $\varphi(\name{G},\name{I}^*)$. Indeed, as $\varphi(\name{G},\name{I}^*)$ is $\Pi^0_1(\name{G},\name{I}^*)$ and $p^*$ forces that $\name{I}^*$ is low over $\name{G}$, it follows that there is a $\Sigma^0_2(\name{G})$ formula $\theta(\name{G})$ that $p^*$ forces is equivalent to $\varphi(\name{G},\name{I}^*)$. Since $n \geq |p^*|$ we have that $\widehat{r},\widehat{r}' \leq p^*$, and so this equivalence is still forced by $\widehat{r}$ and $\widehat{r}'$. Thus, $\widehat{r}$ forces $\theta(\name{G})$, and hence so does $\widehat{r}'$ by Lemma \ref{L:lockspres}. Now it follows that $\widehat{r}'$ forces $\varphi(\name{G},\name{I}^*)$, as desired.
	
	By the uniformity of the low basis theorem, we can fix names $\name{X}_0$ and $\name{X}_1$ and a condition $\widehat{r}'' \leq \widehat{r}'$ forcing that $\name{X}_0 \oplus \name{X}_1$ is low over $\name{G}$ and $\psi(\name{G},\name{I}^*, \name{X}_0,\name{X}_1)$ holds. We may further assume that $\widehat{r}''$ decides, for each $i < 2$, whether or not $\name{X}_i$ is infinite. Since $\widehat{r}''$ forces that $\name{I}^*$ is infinite and $\name{X}_0 \cup \name{X}_1 = \name{I}^* \cap [u,\infty)$, we can fix $i < 2$ such that $\widehat{r}''$ forces that $\name{X}_i$ is infinite. But now consider the $\Qforcing_{p_\Gamma}$-condition $(\widehat{r}'',D_0,D_1,\name{X}_i)$. This is an extension (in $\Qforcing_{p_\Gamma}$) of $(p^*,D_0^*,D_1^*,\name{I}^*)$, and $\widehat{r}''$ forces that $\Psi_{t_i}(\name{G} \oplus (D_i \cup F))(x) \downarrow \simeq 0$ for all finite subset $F$ of $\name{X}_i$ and all $x > z$. By maximality of $M$, this means that $\seq{i,t_i}$ should be in $M$, even though we assumed it was not. This is a contradiction.
	
	We conclude that $\widehat{r}$ actually forces $\neg \varphi(\name{G},\name{I}^*)$, and so some $q \leq \widehat{r}$ must force
	\[
		\neg \psi(\name{G},\name{I}^*,\name{I}^* \cap [u,\infty) \cap \name{A}_0,\name{I}^* \cap [u,\infty) \cap \name{A}_1).
	\]
	In particular, there is an $i < 2$, an $x > z$, and a finite set $F$ such that $q$ forces that $F \subseteq \name{I}^* \cap [u,\infty) \cap \name{A}_i$ and that $\Psi_{t_i}(\name{G} \oplus (D_i \cup F))(x) \downarrow = 1$ and $\name{G}_{n}(x) = j$. Let $E_i = D_i \cup F$ and $E_{1-i} = E_i$, and let $v = \max F$. Then $q$ is the desired extension of $r$, as witnessed by $(q,E_0,E_1,v)$.
\end{proof}

\subsection{Putting it all together}

We are now ready to prove the main theorem of this section, which is Theorem \ref{T:main} for stable $2$-colorings. In fact, we prove following stronger result which clearly implies it.

\begin{theorem}\label{T:multiple}
	Let $\mathcal{G}$ be a hyperarithmetically generic filter on $\Pforcing$. Then for every stable coloring $c : [\omega]^2 \to 2$ hyperarithmetical in $G^\mathcal{G}$, and every finite collection of Turing functionals $\Psi_0,\ldots,\Psi_{s-1}$, there exists an infinite limit-homogeneous set $L$ for $c$ such that $\Psi_t(G^\mathcal{G} \oplus L)$ is not an infinite cohesive set for $G^\mathcal{G}$, for any $t < s$.
\end{theorem}

\begin{proof}
	Let $c$ and $\Psi_0,\ldots,\Psi_{s-1}$ be given. Fix a hyperarithmetical operator $\Gamma$ such that $c = \Gamma(G^\mathcal{G})$. If $c$ has an infinite limit-homogeneous set which is low over $G^\mathcal{G}$, then we can take this to be $L$ and then we are done by Lemma \ref{L:nolow}. So assume otherwise, and choose $p_\Gamma \in \mathcal{G}$ forcing that $\Gamma(\name{G})$ is a stable coloring $[\omega]^2 \to 2$ with no infinite limit-homogeneous set which is low over $\dot{G}$. Define $\name{A}_0$, $\name{A}_1$, and $\Qforcing_{p_\Gamma}$ as in the previous section. Since $\mathcal{G}$ is generic, we may fix a $p^* \in \mathcal{G}$, a $\Qforcing_{p_\Gamma}$-condition $(p^*,D^*_0,D^*_1,\name{I}^*)$, and a maximal subset $M$ of $2 \times s$ as in Lemma \ref{L:Mexists}. We define a sequence of $\Rforcing_{p^*,D_0^*,D_1^*,\name{I}^*}$-conditions
	\[
		(p_0,D_{0,0},D_{0,1},u_0) \geq (p_1,D_{1,0},D_{1,1},u_1) \geq (p_2,D_{2,0},D_{2,1},u_2) \geq \cdots
	\]
	with $p_z \in \mathcal{G}$ for all $z \in \omega$.
	
	If there is an $i < 2$ such that $\seq{i,t} \in M$ for all $t < s$, let $(p_0,D_{0,0},D_{0,1},u_0) = (p^*,D^*_0,D^*_1,\name{I}^*)$. Now given $(p_z,D_{z,0},D_{z,1},u_z)$ for some $z$, apply Lemma \ref{L:infsides} to find an extension $(p_{z+1},D_{z+1,0},D_{z+1,1},u_{z+1})$ with $p_{z+1} \in \mathcal{G}$ and $|D_{z+1,i}| = |D_{z,i}| + 1$ for each $i < 2$. Thus, $L = \bigcup_{z \in \omega} D_{z,i}$ is an infinite limit homogeneous set for $\Gamma(G^\mathcal{G})$, and by assumption, and the definition of $M$, we have $\Psi_t(G^\mathcal{G} \oplus L)(x) \simeq 0$ for all $t < s$ and all sufficiently large $x$. In particular, $\Psi_t(G^\mathcal{G} \oplus L)$ is not an infinite cohesive set for $G^\mathcal{G}$, as desired.
	
	Now suppose that for each $i < 2$ there is at least one $t < s$ with $\seq{i,t} \notin M$. Let $p_0$ be any extension of $p^*$ in $\mathcal{G}$ such that $f^{p_0}(n) = \unlock$ for some $n \in [|p^*|,|p_0|)$, and denote the least such $n$ by $n_0$. Let $D^0_i = D^*_i$ for each $i < 2$, and $u_0 = 0$, so that $(p_0,D_{0,0},D_{0,1},u_0) = (p_0,D^*_0,D^*_1,\name{I}^*)$. Assume next that we have defined $(p_z,D_{z,0},D_{z,1},u_z)$ for some $z$. If $z$ is even, define $(p_{z+1},D_{z+1,0},D_{z+1,1},u_{z+1})$ as in the preceding case, thereby ensuring that $|D_{z+1,i}| = |D_{z,i}| + 1$ for each $i < 2$. Next, suppose $z$ is odd. Assume we have a fixed map $h$ from the odd integers onto the set
	\[
		[(\{0\} \times s) \times (\{1\} \times s) \smallsetminus M^2] \times 3,
	\]
	in which the pre-image of every element in the range is infinite. Say $h(z) = \seq{\seq{0,t_0},\seq{1,t_1},j}$. We then apply Lemma \ref{L:progress} to find $(p_{z+1},D_{z+1,0},D_{z+1,1},u_{z+1})$ extending $(p_z,D_{z,0},D_{z,1},u_z)$ with $p_{z+1} \in \mathcal{G}$ such that for some $i < 2$ and $x > z$ we have that $p_{z+1}$ forces $\Psi_{t_i}(\name{G} \oplus D_{z+1,i})(x) \downarrow = 1$ and $\name{G}_{n_0}(x) = j$.
	
	Now, let $L_i = \bigcup_{z \in \omega} D_{z,i}$ for each $i < 2$, which is an infinite limit-homogeneous set for $\Gamma(G^\mathcal{G})$. If, for each $i < 2$, there is $t_i < s$ such that $\Psi_{t_i}(G^\mathcal{G} \oplus L_i)$ is an infinite cohesive set for $G^\mathcal{G}$, then by genericity of $\mathcal{G}$ and the definition of $M$, it must be that $\seq{i,t_i} \notin M$. For each $j < 3$, there are infinitely many odd numbers $z$ such that $h(z) = \seq{\seq{0,t_0},\seq{1,t_1},j}$, and by construction, for each such $z$, there is an $i < 2$ and an $x > z$ such that $\Psi_{t_i}(G^\mathcal{G} \oplus L_i)(x) \downarrow = 1$ and $G^\mathcal{G}_{n_0}(x) = j$. Denote the least such $i$ by $i_z$. Thus, for each $j < 3$ there must be a $k_j < 2$ such that $i_z = k_j$ for infinitely many $z$ with $h(z) = \seq{\seq{0,t_0},\seq{1,t_1},j}$. Fix $j,j' < 3$ with $j \neq j'$ and $k_j = k_{j'}$, and denote the latter by $i$. Then there are infinitely many $x$ such that $\Psi_{t_i}(G^\mathcal{G} \oplus L_i)(x) \downarrow = 1$ and $G^\mathcal{G}_{n_0}(x) = j$, and infinitely many $x$ such that $\Psi_{t_i}(G^\mathcal{G} \oplus L_i)(x) \downarrow = 1$ and $G^\mathcal{G}_{n_0}(x) = j'$. Thus, $\Psi_{t_i}(G^\mathcal{G} \oplus L_i)$ is not cohesive for $G^\mathcal{G}$, a contradiction.
	
	We conclude that there is an $i < 2$ such that $\Psi_{t}(G^\mathcal{G} \oplus L_i)$ is not an infinite cohesive set for $G^\mathcal{G}$, for any $t < s$, as was to be shown.
\end{proof}

\section{Extensions to arbitrary colorings}\label{S:extensions}

To prove Theorem \ref{T:main} in full generality, we need to modify our construction of the family $G = \seq{G_0,G_1,\ldots}$. Specifically, whereas a $3$-bounded family of functions sufficed to defeat all potential stable $2$-colorings, we will in general need a $(k+1)$-bounded family to defeat all stable $k$-colorings. For this reason, we introduce the following modification of the forcing notion $\Pforcing$ defined earlier.

\begin{definition}
	Let $\mathbb{P}_\omega$ be the notion of forcing whose conditions are tuples $p = (\sigma_0,\ldots,\sigma_{|p|-1},b,f)$ as follows:
	\begin{itemize}
		\item $|p| \in \omega$;
		\item $\sigma_n \in \omega^{<\omega}$ for each $n < |p|$;
		\item $b$ is a function $|p| \to \omega$, and $\max \ran \sigma_n < b(n)$ for all $n < |p|$;
		\item $f$ is a function $|p| \to 3 \cup \{\unlock\}$, and if $f(n) \neq \unlock$ for some $n < |p|$ then $f(n) < b(n)$.
	\end{itemize}
	A condition $q = (\tau_0,\ldots,\tau_{|q|-1},c,g)$ extends $p$, written $q \leq p$, if:
	\begin{itemize}
		\item $|p| \leq |q|$;
		\item $b \preceq c$;
		\item $f \preceq g$;
		\item $\sigma_n \preceq \tau_n$ for all $n < |p|$;
		\item if $f(n) \neq \unlock$ for some $n < |p|$ then $\tau_n(x) = f(n)$ for all $x \in [|\sigma_n|,|\tau_n|)$.
	\end{itemize}
\end{definition}

\noindent We write $\sigma^p_n$, $b^p$, $f^p$ for $\sigma_n$, $b$, and $f$, as before. It is clear that if $\mathcal{G}$ is a sufficiently generic filter on $\mathbb{P}_\omega$ then $G^{\mathcal{G}} = \bigoplus G^{\mathcal{G}}_n$, where again $G^\mathcal{G}_n = \bigcup_{p \in \mathcal{G}, |p| > n} \sigma^p_n$, is now an instance of $\COH_\omega$. Everything else transfers from $\Pforcing$ to $\mathbb{P}_\omega$ analogously, with obvious changes. In particular, this is true of Lemmas \ref{L:lockspres} and \ref{L:nolow}.

Now, fix a hyperarithmetical operator $\Gamma$, and Turing functionals $\Psi_0,\ldots,\Psi_{s-1}$. Suppose $p_\Gamma$ is a $\mathbb{P}_\omega$-condition forcing, for some $k \geq 2$, that $\Gamma(\name{G})$ is a stable coloring $[\omega]^2 \to k$ with no infinite limit-homogeneous set which is low over $\name{G}$. For each $i < k$, let $\name{A}_i$ be a name for the set $\{x \in \omega : \lim_y \Gamma(\name{G})(x,y) = i\}$. We define a suitable modification of the forcing notion $\Qforcing_{p_\Gamma}$.

\begin{definition}
	Let $\mathbb{Q}_{\omega,p_\Gamma}$ be the notion of forcing whose conditions are tuples $(p,D_0,\ldots,D_{k-1},\name{I})$ as follows:
	\begin{itemize}
		\item $p$ is a $\Pforcing$-condition extending $p_\Gamma$;
		\item $D_i$ is a finite set for each $i < k$, and $p$ forces that $D_i \subseteq \name{A}_i$;
		\item $\name{I}$ is a $\mathbb{P}_\omega$-name, and $p$ forces that $\name{I}$ is an infinite set which is low over $\dot{G}$, and $\max \bigcup_{i < k} D_i < \min \name{I}$.	
	\end{itemize}
	A condition $(q,E_0,\ldots,E_{k-1},\name{J})$ extends $(p,D_0,\ldots,D_{k-1},\name{I})$ if:
	\begin{itemize}
		\item $q \leq p$;
		\item $D_i \subseteq E_i$ for each $i < k$;
		\item $q$ forces that $E_i \smallsetminus D_i \subseteq \name{I}$ for each $i < k$, and that $\name{J} \subseteq \name{I}$.
	\end{itemize}
\end{definition}

\noindent We get an analogue of Lemma \ref{L:Mexists}, stated below. The proof is entirely the same.

\begin{lemma}
	The collection of $\Pforcing_\omega$-conditions $p^*$ with the following property is dense below $p_\Gamma$: there exists a $\mathbb{Q}_{\omega,p_\Gamma}$-condition $(p^*,D_0^*,\ldots,D_{k-1}^*,\name{I}^*)$ and a maximal subset $M$ of $k \times s$ such that for all $\seq{i,t} \in M$, $p^*$ forces that there is a $z \in \omega$ such that $\Psi_t(\name{G} \oplus (D_i^* \cup F))(x) \simeq 0$ for all finite sets $F \subseteq \name{I}^*$ and all $x > z$.
\end{lemma}

Fixing $(p^*,D_0^*,\ldots,D_{k-1}^*,\name{I})*$ and $M$ as above, we can define an analogue of the restricted forcing of Definition \ref{D:Rforcing}, and obtain analogues of Lemmas \ref{L:infsides} and \ref{L:progress}. For clarity, we include the definition and statements, and omit the proofs, which carry over from above, mutatis mutandis.

\begin{definition}
	Let $\Rforcing_{\omega,p^*,D_0^*,\ldots,D_{k-1}^*,\name{I}^*}$ be the restriction of $\Qforcing_{\omega,p_\Gamma}$ to conditions extending $(p^*,D_0^*,\ldots,D_{k-1}^*,\name{I}^*)$ of the form $(p,D_0,\ldots,D_{k-1},\name{I}^* \cap [u,\infty))$.
\end{definition}

As before, we write $(p,D_0,\ldots,D_{k-1},u)$ for $(p,D_0,\ldots,D_{k-1},\name{I}^* \cap [u,\infty))$.

\begin{lemma}
	Let $(p,D_0,\ldots,D_{k-1},u)$ be an $\Rforcing_{\omega,p^*,D_0^*,\ldots,D_{k-1}^*,\name{I}^*}$-condition. The collection of $\Pforcing_\omega$-conditions $q$ for which there exists an $\Rforcing_{\omega,p^*,D_0^*,\ldots,D_{k-1}^*,\name{I}^*}$-condition $(q,E_0,\ldots,E_{k-1},v)$ extending $(p,D_0,\ldots,D_{k-1},u)$, and satisfying $|E_i| = |D_i| + 1$ for each $i < k$, is dense below $p$.
\end{lemma}

\begin{lemma}
	Let $(p,D_0,\ldots,D_{k-1},u)$ be an $\Rforcing_{\omega,p^*,D_0^*,\ldots,D_{k-1}^*,\name{I}^*}$-condition., and assume that $b^p(n) = k+1$ and $f^p(n) = \unlock$ for some $n \in [|p^*|, |p|)$. For all $z \in \omega$, $j < 3$, and $\seq{0,t_0}, \ldots,\seq{k-1,t_{k-1}} \in k \times s \smallsetminus M$, the collection of $\mathbb{P}_\omega$-conditions $q$ with the following property is dense below $p$: there exists an $\Rforcing_{\omega,p^*,D_0^*,\ldots,D_{k-1}^*,\name{I}^*}$-condition $(q,E_0,\ldots,E_{k-1},v)$ extending $(p,D_0,\ldots,D_{k-1},u)$ and numbers $i < k$ and $x > z$ such that $q$ forces that $\Psi_{t_i}(\name{G} \oplus E_i)(x) \downarrow = 1$ and $\name{G}_n(x) = j$.
\end{lemma}

Everything can now be put together as in the proof of Theorem \ref{T:multiple} above, to prove the theorem below, from which Theorem \ref{T:main} follows.

\begin{theorem}
	Let $\mathcal{G}$ be a hyperarithmetically generic filter on $\mathbb{P}_\omega$. Then for every $k \geq 2$ and every stable coloring $c : [\omega]^2 \to k$ hyperarithmetical in $G^\mathcal{G}$, and every finite collection of Turing functionals $\Psi_0,\ldots,\Psi_{s-1}$, there exists an infinite limit-homogeneous set $L$ for $c$ such that $\Psi_t(G^\mathcal{G} \oplus L)$ is not an infinite cohesive set for $G^\mathcal{G}$, for any $t < s$.
\end{theorem}

\end{document}